\documentclass[12pt,a4paper]{amsart}

\usepackage{amsmath,amsfonts,amssymb}

\usepackage[hmargin=2cm,vmargin=2cm]{geometry}

\usepackage{hyperref}
\usepackage{nameref,zref-xr}                    

\newcommand{\mbP}{\mathbb P}
\newcommand{\mbZ}{\mathbb Z}
\newcommand{\mbC}{\mathbb C}

\newcommand{\oM}{\overline{\mathcal M}}

\def\oM{{\overline{\mathcal{M}}}}

\def\mbQ{{\mathbb Q}}
\def\d{{\partial}}

\newcommand{\<}{\left<}
\renewcommand{\>}{\right>}
\newcommand{\eps}{\varepsilon}

\newcommand{\Coef}{\mathrm{Coef}}

\DeclareMathOperator{\Res}{Res}
\DeclareMathOperator{\res}{res}
\newcommand{\ext}{\mathrm{ext}}

\newcommand{\rspin}{\text{$r$-spin}}

\newcommand{\mcC}{\mathcal{C}}
\newcommand{\oW}{\overline{\mathcal{W}}}
\newcommand{\mcF}{\mathcal{F}}
\newcommand{\FJRW}{\mathrm{FJRW}}
\newcommand{\mcA}{\mathcal{A}}
\newcommand{\mcO}{\mathcal{O}}
\newcommand{\mcT}{\mathcal{T}}

\newcommand{\threespin}{{\text{$3$-spin}}}
\newcommand{\tcF}{\widetilde{\mathcal{F}}}
\newcommand{\hL}{\widehat{L}}
\newcommand{\ts}{\widetilde{s}}
\newcommand{\mcP}{\mathcal{P}}

\newcommand{\mcL}{\mathcal{L}}
\newcommand{\mcW}{\mathcal{W}}
\newcommand{\mbL}{\mathbb{L}}
\newcommand{\st}{\mathrm{st}}
\newcommand{\tA}{\widetilde{\mathcal{A}}}

\newtheorem{theorem}{Theorem}[section]

\newtheorem{lemma}[theorem]{Lemma}

\newtheorem{conjecture}[theorem]{Conjecture}

\newtheorem*{remark}{Remark}

\newtheorem{question}{Question}

\numberwithin{equation}{section}

\begin{document}

\title[Extended $r$-spin theory and the mirror symmetry for the $A_{r-1}$-singularity]{Extended $r$-spin theory and the mirror symmetry for the $A_{r-1}$-singularity}

\author{Alexandr Buryak}
\address{School of Mathematics, University of Leeds, Leeds, LS2 9JT, United Kingdom}
\email{a.buryak@leeds.ac.uk}
\keywords{Moduli space of curves, Frobenius manifold, singularity, mirror symmetry}
\subjclass[2010]{14H10, 53D45}

\thanks{This project has received funding from the European Union's Horizon 2020 research and innovation programme under the Marie Sk\l odowska-Curie grant agreement No. 797635 and was also supported by grants RFBR-20-01-00579 and RFBR-16-01-00409. We are grateful to V. Goryunov, S. M. Gusein-Zade, O. Karpenkov, K. Saito, S. Shadrin and R. J. Tessler for useful discussions. We also thank A. Basalaev and F. Janda, who noticed a slight inaccuracy in our consideration of the $D_4$-case in the first version of the paper.
}

\begin{abstract}
By a famous result of K.~Saito, the parameter space of the miniversal deformation of the $A_{r-1}$-singularity carries a Frobenius manifold structure. The Landau--Ginzburg mirror symmetry says that, in the flat coordinates, the potential of this Frobenius manifold is equal to the generating series of certain integrals over the moduli space of $r$-spin curves. In this paper we show that the parameters of the miniversal deformation, considered as functions of the flat coordinates, also have a simple geometric interpretation using the extended $r$-spin theory, first considered by T. J. Jarvis, T. Kimura and A. Vaintrob~\cite{JKV01b}, and studied in a recent paper of E.~Clader, R.~J.~Tessler and the author~\cite{BCT17}. We prove a similar result for the singularity~$D_4$ and present conjectures for the singularities~$E_6$ and~$E_8$.
\end{abstract}

\maketitle

\section{Introduction}

The Landau--Ginzburg mirror symmetry conjecture originates from an old physical construction of P.~Berglund and T.~H\"ubsch~\cite{BH93}. Let us very briefly recall the general statement.

\bigskip 

Let $N\ge 1$ and let us fix a matrix $A=(a_{ij})_{1\le i,j\le N}$ with non-negative integer entries~$a_{ij}$. Consider the polynomial $W(x_1,\ldots,x_N)$ and its mirror partner $W^T(x_1,\ldots,x_N)$, defined by
$$
W(x_1,\ldots,x_N):=\sum_{i=1}^N\prod_{j=1}^N x_j^{a_{ij}},\qquad W^T(x_1,\ldots,x_N):=\sum_{i=1}^N\prod_{j=1}^N x_j^{a_{ji}}.
$$ 
Suppose that the polynomial $W$ is quasihomogeneous, has an isolated critical point at the origin and $\det A\ne 0$. Quasihomogeneity means that there exist positive rational numbers $q_1,\ldots,q_N$ such that
$$
W(\lambda^{q_1}x_1,\lambda^{q_2}x_2,\ldots,\lambda^{q_N}x_N)=\lambda W(x_1,\ldots,x_N),
$$
for each $\lambda\in\mbC^*$. There are two theories, associated to the polynomial $W$. They are usually called the A-model and the B-model. 

\bigskip

The A-model is the Fan--Jarvis--Ruan--Witten (FJRW) theory (\cite{FJR13,FJR07,Wit93}) of the pair $(W,G_{W})$, where $G_{W}$ is the maximal group of diagonal symmetries of the polynomial~$W$:
$$
G_{W}:=\left\{\left.(\lambda_1,\ldots,\lambda_N)\in(\mbC^*)^N\right|W(\lambda_1x_1,\ldots,\lambda_Nx_N)=W(x_1,\ldots,x_N)\right\}.
$$
The main object in this theory is the moduli space of $W$-orbicurves. Recall that an orbifold curve $C$ with marked points $p_1,\ldots,p_n$ is a (possibly nodal) Riemann surface~$C$ with orbifold structure at each $p_i$ and each node. Moreover, we require that the local picture at each node is $\{xy=0\}/\mbZ_m$, for some $m\ge 1$, where the action of the group $\mbZ_m$ of $m$-th roots of unity is given by $\zeta_m\cdot(x,y)=(\zeta_m x,\zeta_m^{-1}y)$, $\zeta_m=e^{\frac{2\pi i}{m}}$. For an orbifold curve $C$ denote by $\rho\colon C\to|C|$ the forgetful map to the underlying (coarse, or non-orbifold) curve~$|C|$. A $W$-orbicurve is a marked orbifold curve $(C;p_1,\ldots,p_n)$ together with a collection of orbifold line bundles $L_1,\ldots,L_N$ over~$C$ and isomorphisms
$$
\phi_i\colon\bigotimes_{j=1}^N L_j^{\otimes a_{ij}}\stackrel{\sim}{\to}\rho^*\left(\omega_{|C|}\left(\sum\nolimits_{i=1}^np_i\right)\right),\quad 1\le i\le N,
$$
where $\omega_{|C|}$ is the canonical line bundle on $|C|$. Suppose that the local group at a marked point $p_i$ of $C$ is $\mbZ_{m_i}$, $m_i\ge 1$. Then the line bundles $L_1,\ldots,L_N$ induce a representation $\theta_i\colon\mbZ_{m_i}\to(\mbC^*)^N$. Our $W$-orbicurve is called stable if the underlying marked curve $(|C|;p_1,\ldots,p_n)$ is stable and if for each marked point $p_i$ the representation $\theta_i\colon\mbZ_{m_i}\to(\mbC^*)^N$ is faithful. 

\bigskip

In~\cite{FJR13} the authors proved that the moduli space of stable $W$-orbicurves of genus~$g$ with~$n$ marked points is a smooth compact orbifold. It is denoted by~$\oW_{g,n}$. The moduli space~$\oW_{g,n}$ is not connected. Numerical invariants of the representations $\theta_i$, $1\le i\le n$, give a decomposition of the moduli space $\oW_{g,n}$ into open and closed components. Consider now the case $g=0$. In~\cite{FJR07} the authors constructed a virtual fundamental class on each component of $\oW_{0,n}$ and defined the corresponding intersection number. All these intersection numbers for all components of $\oW_{0,n}$ and for all $n$ can be naturally written as the coefficients of a generating series, which is a formal power series in variables $t_0,\ldots,t_{\mu^T-1}$, $\mu^T\ge 1$, with rational coefficients. Here the number~$\mu^T$ is equal to the dimension of the local algebra
$$
\mcA_{W^T}:=\mcO_{\mbC^N,0}\left/\left(\frac{\d W^T}{\d x_1},\ldots,\frac{\d W^T}{\d x_N}\right)\right.=\mbC[x_1,\ldots,x_N]\left/\left(\frac{\d W^T}{\d x_1},\ldots,\frac{\d W^T}{\d x_N}\right)\right.
$$
of the singularity of $W^T$ at the origin, where by $\mcO_{\mbC^N,0}$ we denote the ring of germs of holomorphic functions on $\mbC^N$ at the origin. The generating series of the intersection numbers is denoted by
$$
\mcF^{\FJRW}_{0,W}(t_0,\ldots,t_{\mu^T-1})\in\mbQ[[t_0,\ldots,t_{\mu^T-1}]].
$$

\bigskip

In~\cite{FJR07} the authors proved that the function $\mcF^{\FJRW}_{0,W}$ satisfies the WDVV equations and, therefore, defines a Frobenius manifold structure in a formal neighbourhood of $0\in\mbC^{\mu^T}$. Frobenius manifolds were introduced and studied in detail by B.~Dubrovin in~\cite{Dub96}. For a more detailed introduction to the FJRW theory, we refer a reader to the original papers~\cite{FJR13,FJR07}.

\bigskip

The B-model is the Saito Frobenius manifold structure on the parameter space of a miniversal deformation of the singularity of the polynomial~$W$. A miniversal deformation (also called a universal unfolding) of the singularity of $W$ is a deformation
\begin{align}
&W_s(x_1,\ldots,x_N)=W(x_1,\ldots,x_N)+s_0+\sum_{i=1}^{\mu-1}s_i\phi_i(x_1,\ldots,x_N),\label{eq:general miniversal deformation}\\
&\phi_i(x_1,\ldots,x_N)\in\mbC[x_1,\ldots,x_N],\quad s_i\in\mbC,\notag
\end{align}    
where the polynomials $\phi_0:=1,\phi_1,\ldots,\phi_{\mu-1}$ form a basis of the local algebra $\mcA_W$ of $W$ at the origin and $\mu$ is the dimension of $\mcA_W$.

\bigskip

The Frobenius manifold structure on the parameter space $\mbC^\mu=\{(s_0,\ldots,s_{\mu-1})|s_i\in\mbC\}$ of the miniversal deformation~\eqref{eq:general miniversal deformation} is constructed in the following way. Consider the deformation $W_s(x_1,\ldots,x_N)$, as a function on~$\mbC^N\times\mbC^\mu$,
$$
W_s(x_1,\ldots,x_N)\in\mcO_{\mbC^N\times\mbC^\mu,0},
$$ 
and consider the ring 
$$
\tA_W:=\mcO_{\mbC^N\times\mbC^\mu}\left/\left(\frac{\d W_s}{\d x_1},\ldots,\frac{\d W_s}{\d x_N}\right)\right..
$$
Via the natural projection $\mbC^N\times\mbC^\mu\to\mbC^\mu$ the ring $\tA_W$ becomes an $\mcO_{\mbC^\mu,0}$-algebra. Moreover, it is a free $\mcO_{\mbC^\mu,0}$-module with the basis $\phi_0(x),\ldots,\phi_{\mu-1}(x)$. Denote by $\mcT_{\mbC^\mu,0}$ the space of germs of sections of the holomorphic tangent bundle $T\mbC^\mu$ to $\mbC^\mu$ at the origin. It is also a free $\mcO_{\mbC^\mu,0}$-module with the basis $\frac{\d}{\d s_0},\ldots,\frac{\d}{\d s_{\mu-1}}$. Let us identify the $\mcO_{\mbC^\mu,0}$-modules $\tA_W$ and $\mcT_{\mbC^\mu,0}$ by identifying the basis elements $\phi_i$ and $\frac{\d}{\d s_i}$, for each $i$. Since $\tA_W$ is an $\mcO_{\mbC^\mu,0}$-algebra, this construction endows the tangent bundle $T\mbC^\mu$ with a multiplication in a neighbourhood of the origin. 

\bigskip

A metric $\frac{1}{2}\sum_{0\le i,j\le\mu-1}g_{ij}(s)ds_ids_j$ on $\mbC^\mu$ is defined in the following way. Define a bilinear form $\<\cdot,\cdot\>$ on the local algebra $\mcA_W$ by
$$
\<p(x),q(x)\>:=\frac{1}{(2\pi i)^N}\int_{\bigcap_{i=1}^N\left\{\left|\frac{\d W}{\d x_i}\right|=\eps\right\}}\frac{p(x)q(x)dx_1\wedge\cdots\wedge dx_N}{\prod_{i=1}^N\frac{\d W}{\d x_i}},\quad p(x),q(x)\in\mcA_W,
$$  
where $\eps$ is a sufficiently small positive number. This bilinear form is symmetric and non-degenerate~\cite[Section~5.11]{AGV85}. K.~Saito~\cite{Sai83} introduced the notion of a primitive form, which is a nowhere vanishing holomorphic form of top degree on $\mbC^\mu$ in a neighbourhood of the origin with certain properties. He proved that for such a form $\zeta$ the metric $g_{ij}(s)$ on $\mbC^\mu$, defined by 
$$
g_{ij}(s):=\frac{1}{(2\pi i)^N}\int_{\bigcap_{i=1}^N\left\{\left|\frac{\d W_s}{\d x_i}\right|=\eps\right\}}\frac{\phi_i(x)\phi_j(x)\zeta}{\prod_{i=1}^N\frac{\d W_s}{\d x_i}},
$$
for sufficiently small $s_i$'s, is flat. Together with the multiplication in $T\mbC^\mu$, constructed above, this metric defines an analytical Frobenius manifold structure on $\mbC^\mu$ in a neighbourhood of the origin. The vector field $\frac{\d}{\d s_0}$ is the unit of it. Let us call this Frobenius manifold the Saito Frobenius manifold. The existence of a primitive form was proved in~\cite{Sai89}. The primitive forms for the simple singularities 
\begin{equation*}
\begin{array}{ll}
A_r\qquad       &W(x)=x^{r+1},\\
D_r             &W(x_1,x_2)=x_1^{r-1}+x_1x_2^2,\\
E_6             &W(x_1,x_2)=x_1^4+x_2^3,\\
E_7             &W(x_1,x_2)=x_1^3x_2+x_2^3,\\
E_8             &W(x_1,x_2)=x_1^5+x_2^3,
\end{array}
\end{equation*}
are given by $\lambda dx_1\wedge\cdots\wedge dx_N$, $\lambda\in\mbC^*$. For a more detailed introduction to theory of the Saito Frobenius manifolds, we refer to the paper~\cite{ST08} and to the book~\cite{Hert02}.

\bigskip

The Landau--Ginzburg mirror symmetry conjecture says that there exists a primitive form~$\zeta$ such that the Saito Frobenius manifold, corresponding to the polynomial $W$, is isomorphic to the Frobenius manifold given by the function $\mcF^\FJRW_{0,W^T}$. A precise description of the necessary primitive form together with the isomorphism is given, for example, in~\cite{HLSW15}. The conjecture is proved in certain cases~\cite{JKV01a,FJR13,KS11,MS16,LLSS17}. A step towards a proof of the conjecture in the general case was made in~\cite{HLSW15}, where the authors managed to prove the conjecture, assuming that certain small set of correlators in the A- and B-models agree (see Theorem~1.2 in~\cite{HLSW15} and the paragraph after it). 

\bigskip

The Landau--Ginzburg mirror symmetry conjecture provides a beautiful link between the singularity theory and the geometry of the moduli spaces of curves. However, one can see that the relation between the A- and B-models, which this conjecture describes, is still not complete. Consider the parameters $s_i(t_*)$ of the miniversal deformation expressed as functions of the flat coordinates. As far as we know, a description of the functions $s_i(t_*)$ and also of the primitive form $\zeta$ in terms of the A-model are not known. Therefore, it is natural to ask the following question.

\begin{question}\label{question}
How to describe the functions $s_i(t_*)$ and the primitive form $\zeta$ in terms of the A-model? 
\end{question}

\smallskip

In this note we answer this question in the case of the $A_{r-1}$-singularity, where $N=1$ and $W(x)=x^r$, $r\ge 2$, $W^T=W$. The mirror symmetry conjecture in this case was proved in~\cite{JKV01a}. The primitive form is trivial, so our question is only about the functions $s_i(t_*)$. We have $\mu=r-1$ and the function $\mcF^\FJRW_{0,W}(t_0,\ldots,t_{r-2})$ can be described as the generating series of the so-called $r$-spin intersection numbers. The $r$-spin theory possesses a certain extension, which was first considered in~\cite{JKV01b} and then studied in~\cite{BCT17} from the point of view of integrable hierarchies. The generating series $\mcF^\ext(t_0,\ldots,t_{r-1})$ of the extended $r$-spin intersection numbers depends on the old variables $t_0,\ldots,t_{r-2}$ and also on an additional variable~$t_{r-1}$. We prove that, up to certain rescaling parameters, the function $s_i(t_0,\ldots,t_{r-2})$ is equal to the coefficient of~$(t_{r-1})^i$ in the series $\frac{\d\mcF^\ext}{\d t_{r-1}}$.

\bigskip

In~\cite{BCT17} E.~Clader, R.~J.~Tessler and the author derived a topological recursion relation for the generating series of the extended $r$-spin intersection numbers with descendents. This equation immediately implies certain WDVV type equations for the function $\mcF^\ext$. We show that the mirror symmetry for the $A_{r-1}$-singularity together with the Saito formulas for the Frobenius manifold structure in the coordinates $s_i$ can be simply derived from these equations.  

\bigskip

We also answer Question~\ref{question} for the singularity $D_4$ and propose conjectural answers for the singularities~$E_6$ and~$E_8$. 

\begin{remark}
In a work in preparation \cite{GKT}, M.~Gross, T.~L.~Kelly and R.~J.~Tessler study open FJRW invariants and provide a similar interpretation of the flat coordinates for the Frobenius manifold for the Landau--Ginzburg models $(\mbC,\mbZ_r,x^r)$ and $(\mbC^2,\mbZ_r\times\mbZ_s,x^r+y^s)$ and their mirrors. In the former case, their results using open $r$-spin invariants are analogous to the results appearing here.
\end{remark}

\subsection*{Plan of the paper}

In Section~\ref{section:LG for A-singularity} we formulate precisely the statement of the Landau-Ginzburg mirror symmetry for the $A_{r-1}$-singularity. The main result of the paper, Theorem~\ref{theorem:main}, which describes the geometric interpretation of the functions $s_i(t_*)$, is contained in Section~\ref{section:main result}. In Section~\ref{section:mirror symmetry from WDVV} we show how to derive the mirror symmetry for the $A_{r-1}$-singularity from the WDVV type equations for the function $\mcF^\ext$. In Section~\ref{section:D4 singularity} we answer Question~\ref{question} for the singularity~$D_4$ and propose conjectural answers for the singularities~$E_6$ and~$E_8$.


\section{Landau-Ginzburg mirror symmetry for the $A_{r-1}$-singularity}\label{section:LG for A-singularity}

In this section we present a more detailed description of the Landau--Ginzburg mirror symmetry for the singularity $A_{r-1}$: $W(x)=x^r$, $r\ge 2$. We would also like to fix a notation for the $r$-th root of $-1$, $\theta_r:=e^{\frac{\pi i}{r}}$, which we will often use in the rest of the paper.

\subsection{A-model}

The FJRW theory of the singularity $W(x)=x^r$ can be equivalently described using the $r$-spin theory (\cite{Chi08,JKV01a}, see also \cite[Section~2]{BCT17}). An orbifold curve $(C;p_1,\ldots,p_n)$ is called $r$-stable, if the coarse underlying marked curve $|C|$ is stable and the isotropy group is $\mbZ_r$ at every marked point and node. Consider a list of integers $0\le\alpha_1,\ldots,\alpha_n\le r-1$. An $r$-spin structure with the twists $\alpha_1,\ldots,\alpha_n$ on an $r$-stable orbifold curve $(C;p_1,\ldots,p_n)$ is an orbifold line bundle $L$ over~$C$ together with an isomorphism 
$$
\phi\colon L^{\otimes r}\stackrel{\sim}{\to}\rho^*\omega_{|C|}\left(-\sum\nolimits_{i=1}^n\alpha_i p_i\right),
$$ 
and such that the isotropy groups at all markings act trivially on the fiber of $L$. Recall that by $\rho\colon C\to|C|$ we denote the forgetful map to the underlying coarse curve $|C|$. The moduli space of $r$-stable orbifold curves of genus $g$ with an $r$-spin structure with the twists $\alpha_1,\ldots,\alpha_n$ is denoted by $\oM^{1/r}_{g;\alpha_1,\ldots,\alpha_n}$. It is non-empty if and only if $2g-2-\sum\alpha_i$ is divisible by $r$, and in this case it is a smooth compact orbifold of complex dimension $3g-3+n$. 

Let us describe now the construction of the virtual fundamental class on $\oM_{g;\alpha_1,\ldots,\alpha_n}^{1/r}$ in the genus $0$ case. We assume that 
\begin{gather}\label{eq:divisibility for rspin}
r\mid\left(\sum\alpha_i+2\right).
\end{gather}
Denote by $\mcC\to\oM^{1/r}_{0;\alpha_1,\ldots,\alpha_n}$ the universal curve and by $\mcL\to\mcC$ the universal line bundle. It is straightforward to check that for any $r$-stable curve $(C;p_1,\ldots,p_n)$ and an $r$-spin structure 
$$
\left(L\to C,\phi\colon L^{\otimes r}\stackrel{\sim}{\to}\rho^*\omega_{|C|}\left(-\sum\alpha_i p_i\right)\right)
$$ 
on $C$ the cohomology group $H^0(C,L)$ vanishes and, therefore, the cohomology group $H^1(C,L)$ has dimension $\frac{\sum_{i=1}^n\alpha_i-(r-2)}{r}$. This implies that $R^1\pi_*\mcL$ is a vector bundle over $\oM_{0;\alpha_1,\ldots,\alpha_n}^{1/r}$ and we denote the dual to it by $\mcW$,
$$
\mcW:=(R^1\pi_*\mcL)^\vee.
$$
It is called the Witten bundle. The top Chern class of it,
\begin{gather}\label{eq:Witten's class}
c_W:=e(\mcW)\in H^{\deg c_W}\left(\oM_{0;\alpha_1,\ldots,\alpha_n}^{1/r},\mbQ\right),\qquad\deg c_W=2\frac{\sum\alpha_i-(r-2)}{r},
\end{gather}
is called the Witten class. It satisfies an important vanishing property which is called the Ramond vanishing: $c_W=0$, if $\alpha_i=r-1$, for some $i$.

The FJRW intersection numbers for the singularity $A_{r-1}$ are also called $r$-spin intersection numbers or $r$-spin correlators. They are obtained by integrating Witten's class against $\psi$-classes on the moduli space $\oM_{0;\alpha_1,\ldots,\alpha_n}^{1/r}$. Denote by $\mbL_i$ the line bundle over $\oM_{0;\alpha_1,\ldots,\alpha_n}^{1/r}$ whose fiber over an $r$-stable curve $C$ is the cotangent space to the coarse curve $|C|$ at the $i$-th marked point. The $r$-spin correlators in genus $0$ are defined by 
\begin{gather}\label{eq:r-spin correlators}
\<\prod_{i=1}^n\tau_{\alpha_i,d_i}\>^\rspin:=r\int_{\oM_{0;\alpha_1,\ldots,\alpha_n}^{1/r}}c_W\prod_{i=1}^n\psi_i^{d_i},\quad d_1,\ldots,d_n\ge 0.
\end{gather}
Because of the Ramond vanishing, this correlator is equal to zero, if $\alpha_i=r-1$, for some~$i$. A correlator $\<\prod\tau_{\alpha_i,d_i}\>^\rspin$ is defined to be zero, if the divisibility condition~\eqref{eq:divisibility for rspin} is not satisfied. Correlators~$\<\prod\tau_{\alpha_i,0}\>^\rspin$ are called primary correlators and also denoted by $\<\prod\tau_{\alpha_i}\>^\rspin$. To be precise, the FJRW intersection numbers for the singularity $A_{r-1}$ coincide with the primary $r$-spin correlators $\<\prod\tau_{\alpha_i}\>^\rspin$, where $0\le\alpha_i\le r-2$.

The FJRW generating series $\mcF^\FJRW_{0,W}$ in our case is also denoted by $\mcF^\rspin$ and defined by
$$
\mcF^\rspin(t_0,\ldots,t_{r-2}):=\sum_{n\ge 3}\sum_{0\le\alpha_1,\ldots,\alpha_n\le r-2}\<\prod_{i=1}^n\tau_{\alpha_i}\>^\rspin\frac{\prod_{i=1}^n t_{\alpha_i}}{n!},
$$
where $t_0,\ldots,t_{r-2}$ are formal variables. From formula~\eqref{eq:Witten's class} for the degree of Witten's class it follows that the series $\mcF^\rspin$ is a polynomial in $t_0,\ldots,t_{r-2}$ which satisfies the following homogeneity condition:
\begin{gather}\label{eq:homogeneity for r-spin}
\mcF^\rspin(\lambda^r t_0,\lambda^{r-1}t_1,\ldots,\lambda^2t_{r-2})=\lambda^{2r+2}\mcF^\rspin(t_0,t_1,\ldots,t_{r-2}),\quad\lambda\in\mbC^*.
\end{gather}
The polynomial $\mcF^\rspin$ satisfies the property $\frac{\d^3\mcF^\rspin}{\d t_0\d t_\alpha\d t_\beta}=\delta_{\alpha+\beta,r-2}$ and also the following system of equations:
\begin{gather}\label{eq:WDVV equations}
\sum_{\mu+\nu=r-2}\frac{\d^3\mcF^\rspin}{\d t_\alpha\d t_\beta\d t_\mu}\frac{\d^3\mcF^\rspin}{\d t_\nu\d t_\gamma\d t_\delta}=\sum_{\mu+\nu=r-2}\frac{\d^3\mcF^\rspin}{\d t_\alpha\d t_\gamma\d t_\mu}\frac{\d^3\mcF^\rspin}{\d t_\nu\d t_\beta\d t_\delta},\quad 0\le\alpha,\beta,\gamma,\delta\le r-2,
\end{gather}
which are called the WDVV equations. Therefore, the function $\mcF^\rspin$ defines a Frobenius manifold structure on $\mbC^{r-1}$ in the coordinates $t_0,\ldots,t_{r-2}$ with the metric $\eta=(\eta_{\alpha\beta})_{0\le\alpha,\beta\le r-2}$, given by $\eta_{\alpha\beta}=\delta_{\alpha+\beta,r-2}$, and the unit vector field $\frac{\d}{\d t_0}$.

It is worth to mention that the polynomial $\mcF^\rspin$ is uniquely determined by the homogeneity condition~\eqref{eq:homogeneity for r-spin}, the WDVV equations~\eqref{eq:WDVV equations} and the following initial conditions:
\begin{gather}\label{eq:initial for rspin}
\<\tau_\alpha\tau_\beta\tau_\gamma\>^\rspin=\delta_{\alpha+\beta+\gamma,r-2},\qquad\<\tau_{r-2}^2\tau_1^2\>^\rspin=\frac{1}{r}.
\end{gather}
This was already shown by E.~Witten in~\cite{Wit93}, but we would also like to mention the paper~\cite[Section~1.2]{PPZ16}, which contains a very short and clear proof of this fact.

\subsection{B-model}

A miniversal deformation of the singularity $W(x)=x^r$ is given by 
$$
W_s(x)=x^r+\sum_{i=0}^{r-2}s_i x^i,\quad s_i\in\mbC.
$$
Let us choose $\zeta=-\theta_r^2 r dx$ to be the primitive form in the Saito construction. So we get the following formula for the metric:
$$
g_{ij}(s)=\frac{1}{2\pi i}\int_{\left|\frac{\d W_s}{\d x}\right|=\eps}\frac{x^{i+j}}{\frac{\d W_s}{\d x}}(-\theta_r^2 r)dx=\theta_r^2 r\Res_{x=\infty}\frac{x^{i+j}}{\frac{\d W_s}{\d x}},\quad 0\le i,j\le r-2.
$$ 

Flat coordinates for the metric $g_{ij}(s)$ can be explicitly constructed in the following way (see e.g.~\cite[page 112]{Dub03}). Consider the series 
$$
k(x):=W_s(x)^{\frac{1}{r}}=x+O(x^{-1}).
$$
Introduce functions $T^\alpha(s_0,\ldots,s_{r-2})$, $1\le \alpha\le r-1$, as the first non-trivial coefficients of the expansion of $x$ in terms of $k(x)$:
$$
x=k+\frac{1}{r}\left(\frac{T^{r-1}(s)}{k}+\frac{T^{r-2}(s)}{k^2}+\ldots+\frac{T^1(s)}{k^{r-1}}\right)+O(k^{-r}).
$$
It is not hard to see that the functions $T^\alpha(s)$ are polynomials in the variables $s_0,\ldots,s_{r-2}$. They are flat coordinates for the metric $g_{ij}(s)$.

The Landau--Ginzburg mirror symmetry conjecture for the singularity $A_{r-1}$ was proved in~\cite{JKV01a}. It says that the change of variables 
\begin{gather*}
t_{\alpha}(T^*)=\theta_r^{r-\alpha}T^{\alpha+1},\quad 0\le\alpha\le r-2,
\end{gather*}
defines an isomorphism between the Saito Frobenius manifold and the Frobenius manifold, given by the potential $\mcF^\rspin$ and the unit vector field $\frac{\d}{\d t_0}$.


\section{Extended $r$-spin theory and the functions $s_i(t_*)$}\label{section:main result}

In this section we describe a certain extension of the $r$-spin theory and prove that the functions~$s_i(t_*)$ are given by the generating series of extended $r$-spin intersection numbers.

The moduli space $\oM_{g;\alpha_1,\ldots,\alpha_n}^{1/r}$ is actually well defined for all integers $\alpha_1,\ldots,\alpha_n$ and there are canonical isomorphisms $\oM_{g;\alpha_1,\ldots,\alpha_n}^{1/r}\cong\oM_{g;\beta_1,\ldots,\beta_n}^{1/r}$, if all differences $\alpha_i-\beta_i$ are divisible by $r$. In~\cite{JKV01b} the authors noticed that the construction of Witten's class on $\oM_{0;\alpha_1,\ldots,\alpha_n}^{1/r}$, described in the previous section, works in the case, when $\alpha_i=-1$ for some $i$, and $0\le\alpha_j\le r-1$ for $j\ne i$. Following~\cite{BCT17}, we refer to this theory as the extended $r$-spin theory. So for all $n\ge 2$ and integers $0\le\alpha_1,\ldots,\alpha_n\le r-1$, satisfying the divisibility condition 
\begin{gather}\label{eq:divisibility for extended rspin}
r\mid\left(\sum\alpha_i+1\right),
\end{gather}
there is well-defined Witten's class 
\begin{gather}\label{eq:degree of extended class}
c_W\in H^{\deg c_W}\left(\oM_{0;-1,\alpha_1,\ldots,\alpha_n}^{1/r},\mbQ\right),\qquad\deg c_W=2\frac{\sum\alpha_i-(r-1)}{r}.
\end{gather}
Extended $r$-spin correlators are defined by
\begin{gather}\label{eq:extended r-spin correlators}
\<\tau_{-1}\prod_{i=1}^n\tau_{\alpha_i,d_i}\>^\rspin:=r\int_{\oM_{0;-1,\alpha_1,\ldots,\alpha_n}^{1/r}}c_W\prod_{i=1}^n\psi_{i+1}^{d_i},\quad d_1,\ldots,d_n\ge 0.
\end{gather}
The correlator~\eqref{eq:extended r-spin correlators} is defined to be zero, if the divisibility condition~\eqref{eq:divisibility for extended rspin} is not satisfied. The Ramond vanishing doesn't hold in the extended theory: for example, in~\cite[Lemma~3.8]{BCT17} the authors showed that $\<\tau_{-1}\tau_1\tau_{r-1}^2\>^\rspin=-\frac{1}{r}$.

Let $t_{r-1}$ be a formal variable and introduce the generating series
$$
\mcF^\ext(t_0,\ldots,t_{r-1}):=\sum_{n\ge 2}\sum_{0\le\alpha_1,\ldots,\alpha_n\le r-1}\<\tau_{-1}\prod_{i=1}^n\tau_{\alpha_i}\>^\rspin\frac{\prod_{i=1}^n t_{\alpha_i}}{n!}.
$$
From formula~\eqref{eq:degree of extended class} it follows that the series $\mcF^\ext$ is a polynomial in $t_0,\ldots,t_{r-1}$ satisfying the following homogeneity property:
\begin{gather}\label{eq:homogeneity for extended}
\mcF^\ext(\lambda^r t_0,\lambda^{r-1}t_1,\ldots,\lambda t_{r-1})=\lambda^{r+1}\mcF^\ext(t_0,t_1,\ldots,t_{r-1}),\quad\lambda\in\mbC^*.
\end{gather}

Let us now consider the polynomial $\mcF^\ext$ as a polynomial in $t_{r-1}$ with the coefficients from~$\mbQ[t_0,\ldots,t_{r-2}]$. Consider also the parameters of the miniversal deformation $s_i(t_*)$, expressed as functions of the flat coordinates $t_0,\ldots,t_{r-2}$. The main result of the paper is the following theorem.
\begin{theorem}\label{theorem:main}
We have
$$
s_i(t_*)=(-r\theta_r)^i\Coef_{t_{r-1}^i}\frac{\d\mcF^\ext}{\d t_{r-1}},\quad 0\le i\le r-2.
$$
\end{theorem}

Before we prove the theorem, let us present an alternative description of the flat coordinates for the Saito Frobenius manifold for the $A_{r-1}$-singularity. Define functions $v_\alpha(s_0,\ldots,s_{r-2})$, $1\le\alpha\le r-1$, by
\begin{gather*}
v_\alpha(s):=-\Res_{x=\infty}\left(W_s(x)^{\frac{\alpha}{r}}\right).
\end{gather*}
It is not hard to see that $v_\alpha(s)$ is a polynomial in $s_0,s_1,\ldots,s_{r-2}$ of the form $v_\alpha(s)=\frac{\alpha}{r}s_{r-\alpha-1}+O(s^2)$. Therefore, the functions $v_1(s),\ldots,v_{r-1}(s)$ can serve as coordinates on $\mbC^{r-1}$ in a neighbourhood of the origin.

\begin{lemma}\label{lemma:another flat coordinates}
We have 
$$
v_\alpha(s)=-\frac{\alpha}{r}T^{r-\alpha}(s),\quad 1\le\alpha\le r-1.
$$
\end{lemma}
\begin{proof}
We only have to check the following identity:
\begin{gather}\label{eq:identity for two flat}
x=k-\sum_{\alpha=1}^{r-1}\frac{v_\alpha(s)}{\alpha}\frac{1}{k^\alpha}+O(x^{-r}).
\end{gather}
For a Laurent series $A=\sum_{i=-\infty}^m a_i(s) x^i$, $a_i(s)\in\mbC[s_0,\ldots,s_{r-2}]$, denote 
$$
A_+:=\sum_{i=0}^m a_ix^i,\qquad A_-:=A-A_+,\qquad \res A:=a_{-1}.
$$
We want to use the results from~\cite[Section~4]{BCT17}. In~\cite[Lemma 4.2]{BCT17} we considered the polynomial $\hL_0=z^r+\sum_{i=0}^{r-2}f_i^{[0]}z^i$ and in~\cite[page 147]{BCT17} we introduced the variables $v_i=\res\left(\hL_0^{i/r}\right)$, $1\le i\le r-1$. We can identify $f_i^{[0]}=s_i$ and $z=x$, then we get $\hL_0=W_s$. By \cite[equation~(4.22)]{BCT17}, we have
\begin{gather}\label{eq:two flat coordinates,eq1}
\frac{r+1}{r}\left(W_s^{\frac{1}{r}}\right)_--\sum_{\alpha=1}^{r-1}\frac{\alpha}{r}\frac{\d}{\d v_\alpha}\res\left(W_s^{\frac{r+1}{r}}\right)W_s^{\frac{\alpha-r}{r}}=O(x^{-r}).
\end{gather}
From the equation before equation~(4.20) in~\cite{BCT17} and~\cite[Lemma 4.2]{BCT17} it follows that
\begin{gather}\label{eq:two flat coordinates,eq2}
v_{r-\alpha}=\frac{\alpha(r-\alpha)}{r+1}\frac{\d}{\d v_\alpha}\res\left(W_s^{\frac{r+1}{r}}\right),\quad 1\le\alpha\le r-1.
\end{gather}
Combining~\eqref{eq:two flat coordinates,eq1} and~\eqref{eq:two flat coordinates,eq2}, we obtain
$$
\left(W_s^{\frac{1}{r}}\right)_--\sum_{\alpha=1}^{r-1}\frac{v_{r-\alpha}}{r-\alpha}\frac{1}{k^{r-\alpha}}=O(x^{-r}).
$$
It remains to note that $\left(W_s^{\frac{1}{r}}\right)_-=k-x$, and identity~\eqref{eq:identity for two flat} becomes clear. The lemma is proved.
\end{proof}

\begin{proof}[Proof of Theorem~\ref{theorem:main}]
The theorem is a consequence of the following stronger statement:
\begin{gather}\label{eq:equation for the main theorem}
W_s(x)=\left.\frac{\d\mcF^\ext}{\d t_{r-1}}\right|_{t_{r-1}=-r\theta_r x}.
\end{gather}
Let us prove it.

Similarly to the proof of Lemma~\ref{lemma:another flat coordinates}, we want to use the results from~\cite[Section~4]{BCT17}. We again identify $f_i^{[0]}=s_i$, $z=x$ and $\hL_0=W_s$. In~\cite[Section 4]{BCT17} we have the variables~$T_i$, $i\ge 1$, and $t^\alpha_d$, $0\le\alpha\le r-1$, $d\ge 0$. Let us set $T_{\ge r+1}=0$, $t^\alpha_{\ge 1}=0$, and denote $t^\alpha:=t^\alpha_0$. We have the following relation \cite[equation~(4.7)]{BCT17}:
\begin{gather*}
t^\alpha=(\alpha+1)(-r)^{\frac{3(\alpha+1)}{2(r+1)}-\frac{1}{2}}T_{\alpha+1},\quad 0\le\alpha\le r-2.
\end{gather*}
Let us use the following notations:
\begin{gather*}
\tcF^\rspin:=\left.\mcF^\rspin\right|_{t_\alpha\mapsto t^\alpha},\qquad \tcF^\ext:=\left.\mcF^\ext\right|_{t_\alpha\mapsto t^\alpha}.
\end{gather*}
Since $\frac{\d^2\tcF^\rspin}{\d T_1\d T_a}=v_a$, $1\le a\le r-1$ \cite[Lemma 4.2]{BCT17}, we get
$$
\frac{\d^2\tcF^\rspin}{\d t^0\d t^{a-1}}a(-r)^{\frac{3(a+1)}{2(r+1)}-1}=v_a,\quad 1\le a\le r-1.
$$
Together with the formulas $\frac{\d^2\tcF^\rspin}{\d t^0\d t^{a-1}}=t^{r-1-a}$ and $v_a=-\frac{a}{r}T^{r-a}$ this implies that 
$$
t^a=\frac{1}{\lambda_r^{r-a}}t_a,\quad\text{where}\quad 0\le a\le r-2\quad\text{and}\quad\lambda_r=\theta_r(-r)^{\frac{3}{2(r+1)}}.
$$

Proposition~4.10 in~\cite{BCT17} says that
$$
\frac{\d\tcF^\ext}{\d t^{r-1}}=\frac{1}{r(-r)^{\frac{r-2}{2(r+1)}}}W_s\left(\frac{\lambda_r}{-r\theta_r}t^{r-1}\right).
$$
We compute
$$
\left.\frac{\d\tcF^\ext}{\d t^{r-1}}\right|_{t^{r-1}=\frac{x}{\lambda_r}}=\frac{\d\tcF^\ext}{\d t^{r-1}}\left(\frac{t_0}{\lambda_r^r},\frac{t_1}{\lambda_r^{r-1}},\ldots,\frac{t_{r-2}}{\lambda_r^2},\frac{x}{\lambda_r}\right)\stackrel{\text{by~\eqref{eq:homogeneity for extended}}}{=}\frac{1}{\lambda_r^r}\left.\frac{\d\mcF^\ext}{\d t_{r-1}}\right|_{t_{r-1}=x}.
$$
Therefore,
$$
\frac{1}{\lambda_r^r}\left.\frac{\d\mcF^\ext}{\d t_{r-1}}\right|_{t_{r-1}=x}=\frac{1}{r(-r)^{\frac{r-2}{2(r+1)}}}W_s\left(\frac{x}{-r\theta_r}\right)\Rightarrow\left.\frac{\d\mcF^\ext}{\d t_{r-1}}\right|_{t_{r-1}=x}=W_s\left(\frac{x}{-r\theta_r}\right),
$$
which proves equation~\eqref{eq:equation for the main theorem}.
\end{proof}


\section{Mirror symmetry as a consequence of the WDVV type equations}\label{section:mirror symmetry from WDVV}

In the previous section we saw that the functions $s_i(t_*)$ appear as the coefficients of the polynomial $\frac{\d\mcF^\ext}{\d t_{r-1}}$. Now we want to show that the Saito formulas for the multiplication and the metric in the coordinates $s_i$ can be naturally deduced from the WDVV type equations for the polynomial~$\mcF^\ext$.

Introduce formal variables $t_{\alpha,d}$, $0\le\alpha\le r-1$, $d\ge 0$, and let $t_{\alpha,0}:=t_\alpha$. Consider the generating series
\begin{align*}
F^\rspin(t_{*,*}):=&\sum_{n\ge 3}\sum_{\substack{0\le\alpha_1,\ldots,\alpha_n\le r-2\\d_1,\ldots,d_n\ge 0}}\<\prod_{i=1}^n\tau_{\alpha_i,d_i}\>^\rspin\frac{\prod_{i=1}^n t_{\alpha_i,d_i}}{n!},\\
F^\ext(t_{*,*}):=&\sum_{n\ge 2}\sum_{\substack{0\le\alpha_1,\ldots,\alpha_n\le r-1\\d_1,\ldots,d_n\ge 0}}\<\tau_{-1}\prod_{i=1}^n\tau_{\alpha_i,d_i}\>^\rspin\frac{\prod_{i=1}^n t_{\alpha_i,d_i}}{n!}.
\end{align*}
In~\cite[equation~(4.26)]{BCT17} the authors proved the following topological recursion relation:
$$
\frac{\d^2F^\ext}{\d t_{\alpha,p+1}\d t_{\beta,q}}=\sum_{\mu+\nu=r-2}\frac{\d^2 F^\rspin}{\d t_{\alpha,p}\d t_{\mu,0}}\frac{\d^2F^\ext}{\d t_{\nu,0}\d t_{\beta,q}}+\frac{\d F^\ext}{\d t_{\alpha,p}}\frac{\d^2F^\ext}{\d t_{r-1,0}\d t_{\beta,q}},\quad0\le\alpha,\beta\le r-1,\quad p,q\ge 0.
$$
These relations can be equivalently written as the following equations between differential $1$-forms:
$$
d\left(\frac{\d F^\ext}{\d t_{\alpha,p+1}}\right)=\sum_{\mu+\nu=r-2}\frac{\d^2 F^\rspin}{\d t_{\alpha,p}\d t_{\mu,0}}d\left(\frac{\d F^\ext}{\d t_{\nu,0}}\right)+\frac{\d F^\ext}{\d t_{\alpha,p}}d\left(\frac{\d F^\ext}{\d t_{r-1,0}}\right),\quad 0\le\alpha\le r-1,\quad p\ge 0.
$$
Taking the exterior derivative of the both sides and setting $p=0$ and $t_{*,\ge 1}=0$, we obtain the following relation:
$$
\sum_{\mu+\nu=r-2}d\left(\frac{\d^2 \mcF^\rspin}{\d t_\alpha\d t_\mu}\right)\wedge d\left(\frac{\d\mcF^\ext}{\d t_\nu}\right)+d\left(\frac{\d\mcF^\ext}{\d t_\alpha}\right)\wedge d\left(\frac{\d\mcF^\ext}{\d t_{r-1}}\right)=0,\quad 0\le\alpha\le r-1,
$$
or, equivalently,
\begin{gather}\label{eq:WDVV type}
\sum_{\mu+\nu=r-2}\frac{\d^3 \mcF^\rspin}{\d t_\alpha\d t_\beta\d t_\mu}\frac{\d^2\mcF^\ext}{\d t_\nu\d t_\gamma}+\frac{\d^2\mcF^\ext}{\d t_\alpha\d t_\beta}\frac{\d^2\mcF^\ext}{\d t_{r-1}\d t_\gamma}=\sum_{\mu+\nu=r-2}\frac{\d^3 \mcF^\rspin}{\d t_\alpha\d t_\gamma\d t_\mu}\frac{\d^2\mcF^\ext}{\d t_\nu\d t_\beta}+\frac{\d^2\mcF^\ext}{\d t_\alpha\d t_\gamma}\frac{\d^2\mcF^\ext}{\d t_{r-1}\d t_\beta},
\end{gather}
where $0\le\alpha,\beta,\gamma\le r-1$. We call these equations the WDVV type equations for the function~$\mcF^\ext$. They already appeared in literature before in the context of open Gromov--Witten theory~\cite{HS12}. For $\beta=r-1$, equation~\eqref{eq:WDVV type} looks as follows:
\begin{gather}\label{eq:WDVV type2}
\frac{\d^2\mcF^\ext}{\d t_\alpha\d t_{r-1}}\frac{\d^2\mcF^\ext}{\d t_{r-1}\d t_\gamma}=\sum_{\mu+\nu=r-2}\frac{\d^3 \mcF^\rspin}{\d t_\alpha\d t_\gamma\d t_\mu}\frac{\d^2\mcF^\ext}{\d t_\nu\d t_{r-1}}+\frac{\d^2\mcF^\ext}{\d t_\alpha\d t_\gamma}\frac{\d^2\mcF^\ext}{\d t^2_{r-1}},
\end{gather}
and is non-trivial only if $0\le\alpha,\gamma\le r-2$. 

Define functions $\ts_i(t_0,\ldots,t_{r-2})$, $0\le i\le r-2$, by 
$$
\ts_i(t_0,\ldots,t_{r-2}):=(-r\theta_r)^i\Coef_{t_{r-1}^i}\left(\frac{\d\mcF^\ext}{\d t_{r-1}}\right).
$$
Of course, we already know that $\ts_i(t_*)=s_i(t_*)$, but we don't want to use it in this section and want to show that the Saito formulas for the Frobenius manifold structure in the coordinates~$\ts_i$ can be derived directly from the WDVV type equations~\eqref{eq:WDVV type2}. 

Let us begin with the metric. We have to prove that
\begin{gather}\label{eq:metric from WDVV}
-\Res_{t_{r-1}=\infty}\left(\frac{\frac{\d^2\mcF^\ext}{\d t_\alpha\d t_{r-1}}\frac{\d^2\mcF^\ext}{\d t_\beta\d t_{r-1}}}{\frac{\d^2\mcF^\ext}{\d t^2_{r-1}}}\right)=-r\delta_{\alpha+\beta,r-2},\quad 0\le\alpha,\beta\le r-2.
\end{gather}
For this we compute
\begin{align}
-\Res_{t_{r-1}=\infty}\left(\frac{\frac{\d^2\mcF^\ext}{\d t_\alpha\d t_{r-1}}\frac{\d^2\mcF^\ext}{\d t_\beta\d t_{r-1}}}{\frac{\d^2\mcF^\ext}{\d t^2_{r-1}}}\right)\stackrel{\text{by~\eqref{eq:WDVV type2}}}{=}&-\Res_{t_{r-1}=\infty}\left(\frac{\d^2\mcF^\ext}{\d t_\alpha\d t_\beta}\right)-\label{eq:mirror from WDVV,1}\\
&-\sum_{\mu+\nu=r-2}\frac{\d^3\mcF^\rspin}{\d t_\alpha\d t_\beta\d t_\mu}\Res_{t_{r-1}=\infty}\left(\frac{\frac{\d^2 \mcF^\ext}{\d t_\nu\d t_{r-1}}}{\frac{\d^2\mcF^\ext}{\d t_{r-1}^2}}\right).\notag
\end{align}
The first term on the right-hand side of this equation is equal to zero, because $\frac{\d^2\mcF^\ext}{\d t_\alpha\d t_\beta}$ is a polynomial in $t_{r-1}$. In order to compute the second term, note that the homogeneity condition~\eqref{eq:homogeneity for extended} implies that the polynomial $\mcF^\ext$ has the form
$$
\mcF^\ext=\<\tau_{-1}\tau_{r-1}^{r+1}\>^\rspin\frac{t_{r-1}^{r+1}}{(r+1)!}+\<\tau_{-1}\tau_{r-2}\tau_{r-1}^{r-1}\>^\rspin t_{r-2}\frac{t_{r-1}^{r-1}}{(r-1)!}+O(t_{r-1}^{r-2}).
$$
Since \cite[proof of Theorem~4.6]{BCT17}
\begin{gather}\label{eq:one-point extended}
\<\tau_{-1}\tau_\gamma\tau_{r-1}^{\gamma+1}\>^\rspin=\frac{\gamma!}{(-r)^\gamma},\quad 0\le \gamma\le r-1,
\end{gather}
we obtain
$$
\frac{\d\mcF^\ext}{\d t_{r-1}}=-\frac{1}{(-r)^r}t_{r-1}^r+\frac{t^{r-2}}{(-r)^{r-2}}t_{r-1}^{r-2}+O(t_{r-1}^{r-3}).
$$
So the second term on the right-hand side of~\eqref{eq:mirror from WDVV,1} is equal to
\begin{align*}
-\frac{\d^3\mcF^\rspin}{\d t_\alpha\d t_\beta\d t_0}\Res_{t_{r-1}=\infty}\left(\frac{\frac{\d^2 \mcF^\ext}{\d t_{r-2}\d t_{r-1}}}{\frac{\d^2\mcF^\ext}{\d t_{r-1}^2}}\right)=&-\delta_{\alpha+\beta,r-2}\Res_{t_{r-1}=\infty}\left(\frac{\frac{t_{r-1}^{r-2}}{(-r)^{r-2}}+O(t_{r-1}^{r-3})}{\frac{t_{r-1}^{r-1}}{(-r)^{r-1}}+O(t_{r-1}^{r-3})}\right)=\\
=&-r\delta_{\alpha+\beta,r-2}.
\end{align*}
Thus, equation~\eqref{eq:metric from WDVV} is proved.

Let us now show that the multiplication in the variables $\ts_i$ is given by the Saito construction. For this we have to check that
\begin{gather}\label{eq:multiplication formula for A}
\sum_{k=0}^{r-2}x^k c^k_{ij}=x^{i+j}\text{ mod }\frac{\d Q}{\d x},
\end{gather}
where by $c^k_{ij}$, $0\le i,j,k\le r-2$, we denote the structure constants of the multiplication in the variables $\ts_*$, and $Q(t_*,x):=\left.\frac{\d\mcF^\ext}{\d t_{r-1}}\right|_{t_{r-1}=-r\theta_rx}$. Denote by $c_{\alpha\beta}^\gamma$, $0\le\alpha,\beta,\gamma\le r-2$, the structure constants of the multiplication in the variables $t_*$. Clearly, $c_{\alpha\beta}^\gamma=\frac{\d^3\mcF^\rspin}{\d t_\alpha\d t_\beta\d t_{r-2-\gamma}}$. We compute (we follow the convention of sum over repeated Greek indices)
\begin{align*}
\sum_{k=0}^{r-2}x^k c^k_{ij}=&\frac{\d t_\alpha}{\d\ts_i}\frac{\d t_\beta}{\d\ts_j}c^\gamma_{\alpha\beta}\sum_{k=0}^{r-2}\frac{\d\ts_k}{\d t_\gamma}x^k=\frac{\d t_\alpha}{\d\ts_i}\frac{\d t_\beta}{\d\ts_j}c^\gamma_{\alpha\beta}\frac{\d Q}{\d t_\gamma}=\\
=&\frac{\d t_\alpha}{\d\ts_i}\frac{\d t_\beta}{\d\ts_j}\sum_{\mu+\nu=r-2}\frac{\d^3\mcF^\rspin}{\d t_\alpha\d t_\beta\d t_\mu}\left.\frac{\d^2\mcF^\ext}{\d t_\nu\d t_{r-1}}\right|_{t_{r-1}=-r\theta_rx}\stackrel{\text{by~\eqref{eq:WDVV type2}}}{=}\\
=&\frac{\d t_\alpha}{\d\ts_i}\frac{\d t_\beta}{\d\ts_j}\left(\frac{\d Q}{\d t_\alpha}\frac{\d Q}{\d t_\beta}+\frac{1}{r\theta_r}\frac{\d Q}{\d x}\left.\frac{\d^2\mcF^\ext}{\d t_\alpha\d t_\beta}\right|_{t_{r-1}=-r\theta_r x}\right)=\\
=&\frac{\d Q}{\d\ts_i}\frac{\d Q}{\d\ts_j}+\frac{\d Q}{\d x}\left(\frac{1}{r\theta_r}\frac{\d t_\alpha}{\d\ts_i}\frac{\d t_\beta}{\d\ts_j}\left.\frac{\d^2\mcF^\ext}{\d t_\alpha\d t_\beta}\right|_{t_{r-1}=-r\theta_rx}\right)=\\
=&x^{i+j}\text{ mod }\frac{\d Q}{\d x},
\end{align*}
which proves formula~\eqref{eq:multiplication formula for A}.


\section{Singularity $D_4$}\label{section:D4 singularity}

In this section we give an answer to Question~\ref{question} for the singularity $D_4=x_1^3+x_1x_2^2$. We also propose conjectural answers for the singularities $E_6$ and $E_8$.

The Landau--Ginzburg mirror symmetry for the polynomial $x_1^3+x_1x_2^2$ is equivalent to the mirror symmetry for the polynomial $x_1^3+x_2^3$. In order to see it, note, first of all, that the singularity $\{x_1^3+x_1x_2^2=0\}\subset\mbC^2$ is just three lines intersecting at the origin. Therefore, a linear change of variables transforms this singularity to the singularity $\{x_1^3+x_2^3=0\}\subset\mbC^2$. Thus, the Saito Frobenius manifolds, corresponding to the polynomials~$x_1^3+x_1x_2^2$ and~$x_1^3+x_2^3$, are isomorphic. The mirror partner to the polynomial $D_4=x_1^3+x_1x_2^2$ is the polynomial $D_4^T=x_1^3x_2+x_2^2$. The potential $\mcF^\FJRW_{0,D_4^T}$ was computed in~\cite{FFJMR16} and in Section~\ref{subsection:A-model for D4} we show that it is equivalent to the FJRW potential for the polynomial $x_1^3+x_2^3$. 

In Section~\ref{subsection:B-model for D4} we describe explictly the B-model for the polynomial $x_1^3+x_2^3$ together with the Landau--Ginzburg mirror symmetry in this case. Then in Section~\ref{subsection:answer for D4} we answer Question~\ref{question} for the polynomial $x_1^3+x_2^3$ and in Section~\ref{subsection:conjecture for E6 and E8} propose conjectures for the singularities $E_6$ and~$E_8$.

\subsection{A-model}\label{subsection:A-model for D4}

Let us consider the more general case of the polynomial
$$
W_{r_1,r_2}(x_1,x_2):=x_1^{r_1}+x_2^{r_2},\quad r_1,r_2\ge 3.
$$
The dimension of the corresponding local algebra $\mcA_{W_{r_1,r_2}}$ is equal to $(r_1-1)(r_2-1)$. The FJRW theory for the polynomial~$W_{r_1,r_2}$ can be described using the $r$-spin theory in the following way. Consider the moduli space $\oM_{0,n}$ of stable curves of genus $0$ with $n$ marked points and denote by $\st\colon\oM_{0;\alpha_1,\ldots,\alpha_n}^{1/r}\to\oM_{0,n}$ the forgetful map, which forgets an $r$-spin structure together with an orbifold structure on an orbifold curve. For $0\le\alpha_1,\ldots,\alpha_n\le r-1$, satisfying the divisibility condition~\ref{eq:divisibility for rspin}, define the cohomology class
\begin{gather}\label{eq:rspin class}
c^\rspin(\alpha_1,\ldots,\alpha_n):=r\cdot\st_*(c_W)\in H^*(\oM_{0,n},\mbQ)
\end{gather}
of degree
\begin{gather}\label{eq:degree of crspin}
\deg c^\rspin(\alpha_1,\ldots,\alpha_n)=2\frac{\sum\alpha_i-(r-2)}{r}.
\end{gather}
The class $c^\rspin(\alpha_1,\ldots,\alpha_n)$ is defined to be zero, if the divisibility condition~\eqref{eq:divisibility for rspin} is not satisfied.

The FJRW intersection numbers for the polynomial $W_{r_1,r_2}$ are given by~\cite[Theorem 4.2.2]{FJR13})
\begin{gather}\label{eq:correlator for W}
\<\prod_{i=1}^n\tau_{\alpha_i,\beta_i}\>^{W_{r_1,r_2}}\hspace{-0.4cm}=\int_{\oM_{0,n}}\hspace{-0.3cm}c^\text{$r_1$-spin}(\alpha_1,\ldots,\alpha_n)c^\text{$r_2$-spin}(\beta_1,\ldots,\beta_n),\hspace{0.3cm}0\le\alpha_i\le r_1-2,\hspace{0.3cm}0\le\beta_i\le r_2-2.
\end{gather}
By~\eqref{eq:degree of crspin} and the divisibility condition~\eqref{eq:divisibility for rspin}, this intersection number is equal to zero unless the following constraints are satisfied:
\begin{gather}\label{eq:selection rules}
\sum_{i=1}^n\left(1-\frac{\alpha_i}{r_1}-\frac{\beta_i}{r_2}\right)=1+\frac{2}{r_1}+\frac{2}{r_2},\qquad \sum_{i=1}^n\alpha_i=r_1-2\text{ mod $r_1$},\qquad \sum_{i=1}^n \beta_i=r_2-2\text{ mod $r_2$}. 
\end{gather}
The FJRW generating series is equal to
$$
\mcF^\FJRW_{0,W_{r_1,r_2}}=\sum_{n\ge 3}\sum_{\substack{0\le\alpha_1,\ldots,\alpha_n\le r_1-2\\0\le\beta_1,\ldots,\beta_n\le r_2-2}}\<\prod_{i=1}^n\tau_{\alpha_i,\beta_i}\>^{W_{r_1,r_2}}\frac{\prod_{i=1}^n t_{\alpha_i,\beta_i}}{n!},
$$
where $t_{\alpha,\beta}$, $0\le\alpha\le r_1-2$, $0\le\beta\le r_2-2$, are formal variables. The unit vector of the corresponding Frobenius manifold is $\frac{\d}{\d t_{0,0}}$ and the metric is
$$
\eta_{\alpha_1,\beta_1;\alpha_2,\beta_2}=\frac{\d^3\mcF^\FJRW_{0,W_{r_1,r_2}}}{\d t_{0,0}\d t_{\alpha_1,\beta_1}\d t_{\alpha_2,\beta_2}}=\delta_{\alpha_1+\alpha_2,r_1-2}\delta_{\beta_1+\beta_2,r_2-2}.
$$

Consider now the case $r_1=r_2=3$. The function $\mcF^\FJRW_{0,W_{3,3}}$ depends on four variables $t_{0,0},t_{1,0},t_{0,1},t_{1,1}$. The metric is antidiagonal in these variables. From constraints~\eqref{eq:selection rules} it follows that there are only two non-trivial $3$-point correlators and only two non-trivial $4$-point correlators:
\begin{gather}\label{eq:initial for D4}
\<\tau_{1,1}\tau_{0,0}\tau_{0,0}\>^{W_{3,3}}=\<\tau_{0,1}\tau_{1,0}\tau_{0,0}\>^{W_{3,3}}=1,\qquad\<\tau_{1,1}\tau^3_{1,0}\>^{W_{3,3}}=\<\tau_{1,1}\tau_{0,1}^3\>^{W_{3,3}}=\frac{1}{3},
\end{gather}
which are computed using equations~\eqref{eq:correlator for W} and~\eqref{eq:initial for rspin}. Using the argument, similar to the one, presented in~\cite[Section~1.2]{PPZ16}, one can show that the function $\mcF^\FJRW_{0,W_{3,3}}$ is uniquely determined by the WDVV equations, the first constraint in~\eqref{eq:selection rules} and the initial conditions~\eqref{eq:initial for D4}. As a result, we obtain
\begin{gather}\label{eq:potential for W33}
\mcF^\FJRW_{0,W_{3,3}}=\frac{1}{2} t_{0,0}^2 t_{1,1}+t_{0,0} t_{1,0}t_{0,1}+\frac{1}{18} t_{1,0}^3 t_{1,1}+\frac{1}{18} t_{0,1}^3 t_{1,1}+\frac{1}{54} t_{1,0} t_{0,1} t_{1,1}^3+\frac{t_{1,1}^7}{68040}.
\end{gather}

Let us compare the potential~\eqref{eq:potential for W33} with the FJRW potential for the singularity $D_4^T$. According to \cite[page 183]{FFJMR16}, we have
$$
\mcF^\FJRW_{0,D_4^T}=\frac{1}{12} t_X^2 t_1 - \frac{1}{4} t_Y^2 t_1 + \frac{1}{12} t_{X^2} t_1^2 + \frac{a}{6} t_X^3 t_{X^2} + \frac{3 a}{2} t_X t_Y^2 t_{X^2} + a^2 t_X^2 t_{X^2}^3 - 3 a^2 t_Y^2 t_{X^2}^3 + \frac{36a^4}{35} t_{X^2}^7,
$$
where $t_1,t_X,t_Y,t_{X^2}$ are formal variables and $a$ is an explicitly computed non-zero constant.  The unit vector field is $\frac{\d}{\d t_1}$. One can directly compute that after the change of variables
\begin{gather*}
t_{0,0}=t_1,\qquad t_{1,0}=\beta t_X+\gamma t_Y,\qquad t_{0,1}=\beta t_X-\gamma t_Y,\qquad t_{1,1}=\delta t_{X^2},
\end{gather*}
with 
$$
\beta=3^{2/3}a^{1/3},\qquad \gamma=3\cdot 3^{1/6}a^{1/3},\qquad\delta=6\cdot 3^{1/3}a^{2/3},
$$
we get 
$$
\mcF^\FJRW_{0,D_4^T}=\frac{1}{36\cdot 3^{1/3}a^{2/3}}\mcF^\FJRW_{0,W_{3,3}}.
$$
Therefore, after a rescaling of the metric the FJRW Frobenius manifolds for the polynomials~$D_4^T$ and~$W_{3,3}$ become isomorphic.

\subsection{B-model}\label{subsection:B-model for D4}

Consider the following miniversal deformation of the singularity $W_{3,3}$:
$$
W_{3,3;s}(x_1,x_2)=x_1^3+x_2^3+s_{1,1}x_1x_2+s_{0,1}x_2+s_{1,0}x_1+s_{0,0},\quad s_{i,j}\in\mbC.
$$
Let us choose the form $\zeta=9dx_1\wedge dx_2$ to be the primitive form, defining the metric $g_{i_1,j_1;i_2,j_2}(s)$:
$$
g_{i_1,j_1;i_2,j_2}(s)=\frac{1}{(2\pi i)^2}\int_{\left|\frac{\d W_{3,3;s}}{\d x_1}\right|=\left|\frac{\d W_{3,3;s}}{\d x_2}\right|=\eps}\frac{x_1^{i_1+i_2}x_2^{j_1+j_2}}{\frac{\d W_{3,3;s}}{\d x_1}\frac{\d W_{3,3;s}}{\d x_2}}9dx_1\wedge dx_2,\quad 0\le i_1,i_2,j_1,j_2\le 1.
$$
Then the matrix of the metric in the variables $s_{0,0},s_{1,0},s_{0,1},s_{1,1}$ is
$$
g(s)=
\begin{pmatrix}
0 & 0 & 0 & 1\\
0 & 0 & 1 & 0\\
0 & 1 & 0 & 0\\
1 & 0 & 0 & \frac{s_{1,1}^2}{9} 
\end{pmatrix}.
$$

A direct computation shows that the change of coordinates 
\begin{gather*}
t_{0,0}(s)=s_{0,0}+\frac{s_{1,1}^3}{54},\qquad t_{1,0}(s)=-s_{1,0},\qquad t_{0,1}(s)=-s_{0,1},\qquad t_{1,1}(s)=s_{1,1},
\end{gather*}
defines an isomorphism between the Saito Frobenius manifold for the singularity $W_{3,3}$ and the Frobenius manifold given by the potential~$\mcF^\FJRW_{0,W_{3,3}}$ and the unit $\frac{\d}{\d t_{0,0}}$.

\subsection{Answer to Question~\ref{question}}\label{subsection:answer for D4}

In the extended $r$-spin theory, by the same formula~\eqref{eq:rspin class}, we can define the class
$$
c^\rspin(-1,\alpha_1,\ldots,\alpha_n)\in H^*(\oM_{0,n+1},\mbQ),\quad0\le\alpha_1,\ldots,\alpha_n\le r-1,
$$
of degree $2\frac{\sum\alpha_i-(r-1)}{r}$. Note also that we have the property \cite[Lemma 3.5]{BCT17} 
$$
c^\rspin(-1,\alpha_1,\ldots,\alpha_n,0)=\pi^*c^\rspin(-1,\alpha_1,\ldots,\alpha_n),
$$
where $\pi\colon\oM_{0,n+2}\to\oM_{0,n+1}$ is the forgetful map, which forgets the last marked point.

Let us define extended intersection numbers for the singularity $W_{3,3}$ by the formula
\begin{gather}\label{eq:extended correlator for D4}
\<\tau_{-1,-1}\prod_{i=1}^n\tau_{\alpha_i,\beta_i}\>^{W_{3,3}}:=\int_{\oM_{0,n+1}}c^\text{$3$-spin}(-1,\alpha_1,\ldots,\alpha_n)c^\text{$3$-spin}(-1,\beta_1,\ldots,\beta_n),\quad 0\le\alpha_i,\beta_i\le 2.
\end{gather}
This correlator is equal to zero unless
\begin{gather}\label{eq:selection rules for extended D4}
\sum_{i=1}^n\left(1-\frac{\alpha_i+\beta_i}{3}\right)=\frac{2}{3},\qquad \sum_{i=1}^n\alpha_i=2\text{ mod $3$},\qquad \sum_{i=1}^n \beta_i=2\text{ mod $3$}. 
\end{gather}

Define a series $\mcP^{W_{3,3}}(t_{0,0},t_{1,0},t_{0,1},t_{1,1},x_1,x_2)$ by
\begin{align*}
&\mcP^{W_{3,3}}(t,x):=\sum_{\substack{n,k_1,k_2\ge 0\\n+k_1+k_2\ge 1}}\frac{x_1^{k_1}}{k_1!}\frac{x_2^{k_2}}{k_2!}\sum_{\substack{0\le\alpha_1,\ldots,\alpha_n\le 1\\0\le\beta_1,\ldots,\beta_n\le 1}}\<\tau_{-1,-1}\tau_{2,2}\tau_{2,0}^{k_1}\tau_{0,2}^{k_2}\prod_{i=1}^n\tau_{\alpha_i,\beta_i}\>^{W_{3,3}}\frac{\prod_{i=1}^n t_{\alpha_i,\beta_i}}{n!}.
\end{align*}
From~\eqref{eq:selection rules for extended D4} it follows that the series $\mcP^{W_{3,3}}$ is actually a polynomial in the variables $t_{\alpha,\beta},x_1,x_2$. The following result is an analog of Theorem~\ref{theorem:main} for the singularity $W_{3,3}$.
\begin{theorem}\label{theorem:extended mirror for D4}
The function $s_{i,j}(t_{*,*})$ is equal to the coefficient of $x_1^ix_2^j$ in the polynomial $\mcP^{W_{3,3}}(t_{*,*},3x_1,3x_2)$.
\end{theorem}
\begin{proof}
We will actually prove a bit stronger statement:
\begin{gather}\label{eq:extended mirror for D4}
W_{3,3;s}(x_1,x_2)=\mcP^{W_{3,3}}(t_{*,*},3x_1,3x_2).
\end{gather}
The proof is by direct computation. For the left-hand side we have
$$
W_{3,3;s}(x)=x_1^3+x_2^3+t_{1,1}x_1x_2-t_{1,0}x_1-t_{0,1}x_2+\left(t_{0,0}-\frac{t_{1,1}^3}{54}\right).
$$

Let us compute the polynomial $\mcP^{W_{3,3}}(t,x)$. From~\eqref{eq:selection rules for extended D4} it follows that it has the form
$$
\mcP^{W_{3,3}}(t,x)=\alpha_1\frac{x_1^3}{6}+\alpha_2\frac{x_2^3}{6}+\alpha_3 t_{1,1}x_1x_2+\alpha_4 t_{1,0}x_1+\alpha_5 t_{0,1}x_2+\left(\alpha_6t_{0,0}+\alpha_7\frac{t_{1,1}^3}{6}\right),\quad\alpha_1,\ldots,\alpha_7\in\mbQ.
$$
We compute
\begin{align*}
&\alpha_1=\alpha_2=\<\tau_{-1,-1}\tau_{2,2}\tau_{2,0}^3\>^{W_{3,3}}=\int_{\oM_{0,5}}c^{\text{$3$-spin}}(-1,2,2,2,2)c^{\text{$3$-spin}}(-1,2,0,0,0)=\\
&\hspace{1.42cm}=\<\tau_{-1}\tau_2^4\>^{\text{$3$-spin}}\stackrel{\text{by \eqref{eq:one-point extended}}}{=}\frac{2}{9},\\
&\alpha_3=\<\tau_{-1,-1}\tau_{2,2}\tau_{2,0}\tau_{0,2}\tau_{1,1}\>^{W_{3,3}}=\int_{\oM_{0,5}}c^\threespin(-1,2,2,0,1)c^\threespin(-1,2,0,2,1).
\end{align*}
Denote by $\pi_3\colon\oM_{0,5}\to\oM_{0,4}$ and $\pi_4\colon\oM_{0,5}\to\oM_{0,4}$ the forgetful maps which forget the third and fourth marked points, respectively. Then
$$
c^\threespin(-1,2,2,0,1)=\pi_4^*c^\threespin(-1,2,2,1)=-\frac{1}{3}\pi_4^*\psi_1=-\frac{1}{3}(\psi_1-D_{1,4}),
$$
where $D_{i,j}$ denotes the cohomology class, Poincar\'e dual to the divisor in $\oM_{0,5}$, whose generic point is a nodal curve made of two bubbles containing the marked points labeled by~$\{i,j\}$ and~$\{1,2,3,4,5\}\backslash\{i,j\}$, respectively. Similarly, 
$$
c^\threespin(-1,2,0,2,1)=-\frac{1}{3}\pi_3^*\psi_1=-\frac{1}{3}(\psi_1-D_{1,3}).
$$
As a result,
$$
\alpha_3=\int_{\oM_{0,5}}c^\threespin(-1,2,2,0,1)c^\threespin(-1,2,0,2,1)=\frac{1}{9}\int_{\oM_{0,5}}(\psi_1-D_{1,4})(\psi_1-D_{1,3})=\frac{1}{9}.
$$
For the constants $\alpha_4$ and $\alpha_5$ we get
$$
\alpha_4=\alpha_5=\<\tau_{-1,-1}\tau_{2,2}\tau_{2,0}\tau_{1,0}\>^{W_{3,3}}=\<\tau_{-1}\tau_2^2\tau_1\>^\threespin\stackrel{\text{by~\eqref{eq:one-point extended}}}{=}-\frac{1}{3}.
$$
We continue with $\alpha_6$:
$$
\alpha_6=\<\tau_{-1,-1}\tau_{2,2}\tau_{0,0}\>^{W_{3,3}}=1.
$$

For the constant $\alpha_7$ we compute
$$
\alpha_7=\<\tau_{-1,-1}\tau_{2,2}\tau_{1,1}^3\>^{W_{3,3}}=\int_{\oM_{0,5}}c^\threespin(-1,2,1,1,1)^2.
$$
Consider the $S_3$-action on $\oM_{0,5}$ induced by permutations of the last three marked points. Clearly, 
$$
c^\threespin(-1,2,1,1,1)\in H^2(\oM_{0,5},\mbQ)^{S_3}.
$$
Let us compute this class. The basis of the group $H^2(\oM_{0,5},\mbQ)^{S_3}$ is given by
$$
D_1=D_{1,2},\qquad D_2=D_{3,4}+D_{3,5}+D_{4,5},\qquad D_3=D_{2,3}+D_{2,4}+D_{2,5}.
$$ 
The intersection matrix $\left(\int_{\oM_{0,5}}\psi_i D_j\right)_{1\le i,j\le 3}$ is non-degenerate. Using also the integrals
$$
\int_{\oM_{0,5}}\psi_ic^\threespin(-1,2,1,1,1)=-\frac{1}{3}\delta_{i,3},\quad i=1,2,3,
$$ 
computed using the topological recursion relations in the extended $r$-spin theory~\cite[Lemma~3.6]{BCT17}, we conclude that $c^\threespin(-1,2,1,1,1)=-\frac{1}{3}D_1=-\frac{1}{3}D_{1,2}$. Therefore,
$$
\alpha_7=\int_{\oM_{0,5}}c^\threespin(-1,2,1,1,1)^2=\frac{1}{9}\int_{\oM_{0,5}}D_{1,2}^2=-\frac{1}{9}.
$$
Thus, the polynomial $\mcP^{W_{3,3}}(t,x)$ is equal to
$$
\mcP^{W_{3,3}}(t,x)=\frac{x_1^3}{27}+\frac{x_2^3}{27}+t_{1,1}\frac{x_1x_2}{9}-t_{1,0}\frac{x_1}{3}-t_{0,1}\frac{x_2}{3}+\left(t_{0,0}-\frac{t_{1,1}^3}{54}\right),
$$
and we can immediately see that the polynomial $W_{3,3;s}(x)$ is obtained from it by the rescaling $x_i\mapsto 3x_i$. The theorem is proved.
\end{proof}

\subsection{Conjecture for the singularities $E_6$ and $E_8$}\label{subsection:conjecture for E6 and E8}

The singularities $E_6$ and $E_8$ coincide with the singularities 
$$
W_{4+m,3}(x_1,x_2)=x_1^{4+m}+x_2^3,
$$
for $m=0$ and $m=1$, respectively. Similarly to the case of the singularity $W_{3,3}$, let us define functions $\mcP^{E_{6+2m}}(t,x)$, $m=0,1$, by
$$
\mcP^{E_{6+2m}}(t,x):=\sum_{\substack{n,k_1,k_2\ge 0\\n+k_1+k_2\ge 1}}\frac{x_1^{k_1}}{k_1!}\frac{x_2^{k_2}}{k_2!}\sum_{\substack{0\le\alpha_1,\ldots,\alpha_n\le 2+m\\0\le \beta_1,\ldots,\beta_n\le 1}}\<\tau_{-1,-1}\tau_{3+m,2}\tau_{3+m,0}^{k_1}\tau_{0,2}^{k_2}\prod_{i=1}^n\tau_{\alpha_i,\beta_i}\>^{W_{4+m,3}}\frac{\prod_{i=1}^n t_{\alpha_i,\beta_i}}{n!}.
$$
The Landau--Ginzburg mirror symmetry for the singularities $E_6$ and $E_8$, as well as for all simple singularities, was proved in~\cite{FJR13}. Denote by $s^{E_{6+2m}}_{i,j}(t)$, $0\le i\le 2+m$, $0\le j\le 1$, the corresponding transformation between the flat coordinates and the parameters of the miniversal deformation.
\begin{conjecture}
We have
$$
s^{E_{6+2m}}_{i,j}(t_{*,*})=\Coef_{x_1^ix_2^j}\mcP^{E_{6+2m}}(t_{*,*},-(4+m)\theta_{4+m}x_1,-3\theta_3x_2).
$$
\end{conjecture}

\end{document}